\documentclass[paper=a4,english,fontsize=11pt,parskip=half,abstract=true]{scrartcl}
\usepackage{babel}
\usepackage[utf8]{inputenc}
\usepackage[T1]{fontenc}
\usepackage[left=20mm,right=20mm,top=30mm,bottom=30mm]{geometry}
\usepackage{amsmath}
\usepackage{amsthm}
\usepackage{amssymb}
\usepackage{thmtools}
\usepackage{mathtools}
\mathtoolsset{centercolon} 
\usepackage[bookmarks=true,
            pdftitle={On the complement of a union of cosets},
            pdfauthor={Benjamin Sambale},
            pdfkeywords={union of cosets, coset cover, irredundant cover, Helly property},
            pdfstartview={FitH}]{hyperref}

\newtheorem{Thm}{Theorem} 

\newtheorem{Lem}[Thm]{Lemma}

\theoremstyle{remark}

\theoremstyle{definition}

\numberwithin{equation}{section}

\allowdisplaybreaks[1]

\renewcommand{\phi}{\varphi}

\newcommand{\ZZ}{\mathbb{Z}}

\newcommand{\QQ}{\mathbb{Q}}

\newcommand{\NN}{\mathbb{N}}
\newcommand{\FF}{\mathbb{F}}

\title{On the complement of a union of cosets}
\author{Benjamin Sambale\footnote{Institut für Algebra, Zahlentheorie und Diskrete Mathematik, Leibniz Universität Hannover, Welfengarten 1, 30167 Hannover, Germany,
\href{mailto:sambale@math.uni-hannover.de}{sambale@math.uni-hannover.de}}}
\date{\today}

\begin{document}
\frenchspacing
\maketitle
\renewcommand{\sectionautorefname}{Section}

\begin{abstract}\noindent
Let $g_1H_1,\ldots,g_nH_n$ be left cosets of subgroups of a group $G$ whose union $U$ is a proper subset of $G$. We prove that $G$ is the union of at most $2^n$ left translates of $A:=G\setminus U$. For finite $G$ this yields $|A|\ge|G|/2^n$, which settles a conjecture of T\u{a}rn\u{a}uceanu and the author. Equality $|A|=|G|/2^n$ can only hold when $A$ is a left coset of $K:=H_1\cap\ldots\cap H_n$. Moreover, every subgroup occurring in an irredundant cover of an arbitrary group by $n$ cosets has index at most $2^{n-1}$. This improves a lemma of Neumann. 
\end{abstract}


\section{Introduction}\label{sec:intro}

A finite group $G$ is never the union of two proper subgroups $H_1$ and $H_2$, and an easy count shows more precisely that $H_1\cup H_2$ misses at least a quarter of the elements of $G$. Results of this kind have a long history. Cameron and Cohen~\cite{CameronCohen} showed that at least $|H|$ elements lie outside a union of conjugates of a proper subgroup $H\le G$, and the groups covered by a prescribed number of proper subgroups have been classified in many cases (see the survey~\cite{Bhargava}).

The general question behind these results is the following. Let $C_1,\ldots,C_n$ be cosets of subgroups of a finite group $G$, and suppose that 
\[
	A:=G\setminus(C_1\cup\ldots\cup C_n)
\]
is nonempty. How small can $A$ be? 
In~\cite{SambaleUnion} T\u{a}rn\u{a}uceanu and the author proved the bound $|A|\ge|G|/(2n!)$ and conjectured that
\begin{equation}\label{eq:conj}
	|A|\ge\frac{|G|}{2^n}
\end{equation}
(this is Problem~21.115 in the Kourovka notebook~\cite{Kourovka21}). 
I learned from Stefanos Aivazidis that an AI-generated proof of~\eqref{eq:conj} has been announced on GitHub~\cite{Danny}.
The aim of this paper is to prove the following stronger result for arbitrary (possibly infinite) groups.

\begin{Thm}\label{thm:A}
Let $G$ be a group with (left or right) cosets $C_1,\ldots,C_n$ of subgroups such that $A:=G\setminus(C_1\cup\ldots\cup C_n)\ne\varnothing$. Then there exist $t_1,\ldots,t_m\in G$ with $G=t_1A\cup\ldots\cup t_mA$ and $m\le2^n$.
\end{Thm}

For finite $G$ the union bound $|G|\le m|A|\le2^n|A|$ yields \eqref{eq:conj}. It has been pointed out in \cite{SambaleUnion} that \eqref{eq:conj} is best possible for $G=\FF_2^n$ and $C_i:=\{x\in G:x_i=1\}$.  
The constant $2^n$ in \autoref{thm:A} is sharp for infinite groups as well: In $G=\ZZ^n$ let $C_i:=\{x\in G:x_i\ \text{odd}\}$. Then $A=(2\ZZ)^n$ is a subgroup of index $2^n$, so exactly $2^n$ translates of $A$ are needed to cover $G$.

The next result describes the equality case for finite $G$ in general.
We may assume w.l.o.g. that the $C_i$ are left cosets since $Hg=g(g^{-1}Hg)$ for $g\in G$ and $H\le G$. 

\begin{Thm}\label{thm:C}
Let $G$ be a finite group with left cosets $C_i=g_iH_i$ of subgroups $H_i\le G$ such that $A=G\setminus(C_1\cup\ldots\cup C_n)\ne\varnothing$. Let $K:=H_1\cap\ldots\cap H_n$. Then
\[
	|A|=\frac{|G|}{2^n}
	\quad\Longleftrightarrow\quad
	A=aK\ \text{for some }a\in G\ \text{and}\ |G:K|=2^n.
\]
In this case the subgroups $H_1,\ldots,H_n$ are pairwise distinct.
\end{Thm}

Now we turn to the situation where $A=\varnothing$. 
Recall that a \emph{cover} $G=C_1\cup\ldots\cup C_n$ by cosets is called \emph{irredundant} if no $C_i$ can be omitted. A classical lemma of Neumann~\cite{Neumann1954} asserts that all subgroups occurring in an irredundant cover have finite index, and that $|G:H_i|\le n$ holds for at least one $i$. In~\cite{Neumann1954b} he bounded \emph{all} the indices by $|G:H_i|\le u_n$, where $u_1=1$ and $u_{k+1}=u_k^2+u_k$, a bound of order $c^{2^{n-1}}$ with $c\approx1.5979$. We improve this doubly exponential bound to the sharp exponential bound $2^{n-1}$.

\begin{Thm}\label{thm:B}
Let $G=C_1\cup\ldots\cup C_n$ be an irredundant cover of a group $G$ by left cosets $C_i=g_iH_i$. Then $|G:H_i|\le2^{n-1}$ for $i=1,\ldots,n$. This bound is best possible for every $n$. 
\end{Thm}

We remark that Tomkinson~\cite{Tomkinson1987} proved the sharp bound $|G:H_1\cap\ldots\cap H_n|\le n!$ in the situation of \autoref{thm:B}.

Throughout, $[n]:=\{1,\ldots,n\}$.

\section{A Helly lemma for unions of equivalence classes}\label{sec:helly}

Let $\mathcal M_n$ denote the space of multilinear polynomials in $\QQ[X_1,\ldots,X_n]$, that is, the $\QQ$-span of the $2^n$ monomials $\prod_{i\in S}X_i$ with $S\subseteq[n]$. These monomials are linearly independent, so
\begin{equation}\label{eq:dim}
	\dim_\QQ\mathcal M_n=2^n.
\end{equation}
The following lemma establishes a Helly property for certain unions of equivalence classes. 

\begin{Lem}\label{lem:helly}
Let $X$ be a set equipped with equivalence relations $\sim_1,\ldots,\sim_n$ and let $U_1,\ldots,U_m\subseteq X$ be sets of the form
\[
	U_j=Q_{j,1}\cup\ldots\cup Q_{j,n},
\]
where $Q_{j,i}$ is a $\sim_i$-class. If $U_1\cap\ldots\cap U_m=\varnothing$ while no proper subfamily has empty intersection, then $m\le2^n$. If moreover $\sim_i=\sim_j$ for some $i\ne j$, then $m<2^n$.
\end{Lem}

\begin{proof}
By minimality there is, for every $j\in[m]$, an element
\[
	x_j\in\bigcap_{k\ne j}U_k,
\]
and $x_j\notin U_j$ because the total intersection is empty. Only finitely many equivalence classes are relevant, namely the $Q_{j,i}$ and the classes of the $x_j$. Hence for each $i$ we may choose an injection $\lambda_i$ from the set of $\sim_i$-classes occurring here into $\QQ$. Put
\[
	\alpha_{j,i}:=\lambda_i(Q_{j,i}),\qquad
	\beta_{j,i}:=\lambda_i([x_j]_{\sim_i}),\qquad
	\alpha_j:=(\alpha_{j,1},\ldots,\alpha_{j,n})\in\QQ^n.
\]
Since $x_j\notin U_j$, the element $x_j$ lies in none of the classes $Q_{j,1},\ldots,Q_{j,n}$, whence
\begin{equation}\label{eq:noneq}
	\beta_{j,i}\ne\alpha_{j,i}\qquad(1\le i\le n).
\end{equation}
On the other hand $x_j\in U_k$ for $k\ne j$, so $x_j$ lies in $Q_{k,i}$ for at least one $i$ and therefore
\begin{equation}\label{eq:eq}
	\beta_{j,i}=\alpha_{k,i}\quad\text{for some }i\in[n].
\end{equation}
Now define
\[
	p_j:=\prod_{i=1}^n(X_i-\beta_{j,i})\in\mathcal M_n\qquad(1\le j\le m).
\]
By~\eqref{eq:noneq} we have $p_j(\alpha_j)\ne0$, and by~\eqref{eq:eq} we have $p_j(\alpha_k)=0$ for $k\ne j$. Evaluating a vanishing linear combination $\sum_jc_jp_j=0$ at $\alpha_k$ isolates $c_kp_k(\alpha_k)$ and thus gives $c_k=0$. Hence $p_1,\ldots,p_m$ are linearly independent in $\mathcal M_n$ and $m\le2^n$ by~\eqref{eq:dim}.

Suppose finally that $\sim_i=\sim_j$ with $i\ne j$. Then both relations have the same classes, so we may choose $\lambda_i=\lambda_j$ and obtain $\beta_{l,i}=\beta_{l,j}$ for every $l$. Consequently each $p_l$ contains the factor $(X_i-\beta_{l,i})(X_j-\beta_{l,i})$ and is therefore symmetric in $X_i$ and $X_j$. 
Now the $p_l$ cannot span $\mathcal M_n$, so $m<2^n$. 
\end{proof}

\section{Union of translates}\label{sec:main}

Let $G$, $H_i$, $C_i=g_iH_i$ and $A$ be as in \autoref{sec:intro}. Set $U:=C_1\cup\ldots\cup C_n$. For each $i$ let $\sim_i$ be the equivalence relation on $G$ whose classes are the left $H_i$-cosets. Every left translate of $U$ then has the shape required in \autoref{lem:helly}, namely
\begin{equation}\label{eq:tU}
	tU=\bigcup_{i=1}^ntg_iH_i\qquad(t\in G).
\end{equation}
Readers interested only in finite groups may jump directly to the proof of \autoref{thm:A} below. For infinite $G$ an additional ingredient is required.

For $\varnothing\ne S\subseteq[n]$ we write $H_S:=\bigcap_{i\in S}H_i$. 
Set $\mathcal L:=\{H_S:\varnothing\ne S\subseteq[n]\}$, a family of subgroups of $G$ closed under intersection.

\begin{Lem}\label{lem:elem}
Let $D=dL$ and $D'=d'L'$ be left cosets with $L,L'\in\mathcal L$. Either $D\cap D'=\varnothing$ or $D\cap D'$ is a left coset of $L\cap L'\in\mathcal L$. In the latter case $D\cap D'=D$ if and only if $L\subseteq L'$.
\end{Lem}

\begin{proof}
If $y\in D\cap D'$, then $D=yL$ and $D'=yL'$, hence $D\cap D'=y(L\cap L')$. This equals $D=yL$ if and only if $L\cap L'=L$.
\end{proof}

\begin{Lem}\label{lem:finite}
Suppose that $1\notin U$. Then there is a finite subset $T\subseteq G$ with $\bigcap_{t\in T}tU=\varnothing$.
\end{Lem}

\begin{proof}
For $L\in\mathcal L$ let the \emph{height} $h(L)$ be the largest $k\ge0$ such that there is a chain $L=L_0\supsetneq L_1\supsetneq\ldots\supsetneq L_k$ in $\mathcal L$. To a finite list $\mathcal D=(D_1,\ldots,D_r)$ of left cosets $D_\rho=d_\rho L_\rho$ with $L_\rho\in\mathcal L$ we attach the weight
\[
	w(\mathcal D):=\sum_{\rho=1}^r(n+1)^{h(L_\rho)}\in\NN.
\]
Note that $w(\mathcal D)=0$ if and only if $\mathcal D$ is the empty list.

We construct finite sets $T$ together with lists $\mathcal D$ whose union is $W:=\bigcap_{t\in T}tU$, starting with $T=\{1\}$ and $\mathcal D=(g_1H_1,\ldots,g_nH_n)$. Assume $T$ and $\mathcal D$ have been constructed and $W\ne\varnothing$, so that $\mathcal D$ is nonempty. Choose $x\in D_1$ and put $T':=T\cup\{x\}$ and $W':=W\cap xU$. Since $1\notin U$ we have $x\notin xU$. We build a list $\mathcal D'$ for $W'$ as follows. Every $D_\rho$ with $D_\rho\subseteq xU$ is kept. Every other $D_\rho$ is replaced by the nonempty sets among $D_\rho\cap xg_iH_i$ with $i\in[n]$, whose union is $D_\rho\cap xU$. Each of these is a proper subset of $D_\rho$, for otherwise $D_\rho\subseteq xg_iH_i\subseteq xU$. Hence by \autoref{lem:elem} it is a left coset of $L_\rho\cap H_i\in\mathcal L$ with $L_\rho\cap H_i\subsetneq L_\rho$, so its height is at most $h(L_\rho)-1$. Consequently the contribution of $\rho$ to the weight drops from $(n+1)^{h(L_\rho)}$ to at most $n(n+1)^{h(L_\rho)-1}<(n+1)^{h(L_\rho)}$. The coset $D_1$ is of the second kind, because $x\in D_1$ but $x\notin xU$. Therefore $w(\mathcal D')<w(\mathcal D)$.

A strictly decreasing sequence of nonnegative integers is finite, so after finitely many steps the construction must halt, which happens only when $W=\varnothing$.
\end{proof}

\begin{proof}[Proof of \autoref{thm:A}]
Fix $a\in A$ and replace $A$, $U$ and every $C_i$ by their left translates by $a^{-1}$. This changes neither the subgroups $H_i$ nor the assertion, so we may assume $1\in A$, that is, $1\notin U$. By \autoref{lem:finite} there is a finite $T\subseteq G$ with $\bigcap_{t\in T}tU=\varnothing$ (for finite $G$, this is obvious since $t\notin tU$). Choose $T$ inclusion-minimal with this property and write $T=\{t_1,\ldots,t_m\}$. By~\eqref{eq:tU} the sets $t_1U,\ldots,t_mU$ satisfy the hypotheses of \autoref{lem:helly} with $X=G$, so $m\le2^n$. Taking complements and using $G\setminus t_jU=t_jA$ we obtain
\[
	G=t_1A\cup\ldots\cup t_mA.\qedhere
\]
\end{proof}

\section{The equality case}\label{sec:equality}

We now analyze when~\eqref{eq:conj} is an equality and prove \autoref{thm:C}. 

The proof of \autoref{thm:C} uses the polynomials of \autoref{lem:helly} in the concrete situation at hand, so we make them explicit. Let $G$ be finite and choose for every $i\in[n]$ an injection $\lambda_i$ from the set of left $H_i$-cosets into $\QQ$. For $y\in G$ and $t\in G$ put
\begin{equation}\label{eq:pyalpha}
	p_y:=\prod_{i=1}^n\bigl(X_i-\lambda_i(yH_i)\bigr)\in\mathcal M_n,
	\qquad
	\alpha_t:=\bigl(\lambda_1(tg_1H_1),\ldots,\lambda_n(tg_nH_n)\bigr)\in\QQ^n.
\end{equation}
Since the $\lambda_i$ are injective, $p_y(\alpha_t)=0$ holds if and only if $yH_i=tg_iH_i$ for some $i$, that is, if and only if $y\in tU$. Hence
\begin{equation}\label{eq:vanish}
	p_y(\alpha_t)=0\iff y\in tU,
	\qquad
	p_y(\alpha_t)\ne0\iff y\in tA.
\end{equation}

\begin{proof}[Proof of \autoref{thm:C}]
If $A=aK$ and $|G:K|=2^n$, then $|A|=|K|=|G|/2^n$. Conversely, assume $|A|=|G|/2^n$. Fix $a\in A$ and replace all $C_i$ by $a^{-1}C_i$. This changes neither $|A|$ nor the subgroups $H_i$, and it replaces $A$ by $a^{-1}A\ni1$. It therefore suffices to prove $A=K$ under the additional assumption $1\in A$. Since $K\le H_i$, we have $C_iK=g_iH_iK=g_iH_i=C_i$ for every $i$. This gives $UK=U$ and $AK=A$. As $1\in A$, we conclude $K\subseteq A$.

Choose $T\subseteq G$ inclusion-minimal with $\bigcap_{t\in T}tU=\varnothing$, which is possible by \autoref{lem:finite}. Equivalently, $G=\bigcup_{t\in T}tA$ and no proper subfamily covers $G$. By~\eqref{eq:tU} and \autoref{lem:helly} we have $|T|\le2^n$, whereas
\[
	|G|\le\sum_{t\in T}|tA|=|T|\,|A|=|T|\,\frac{|G|}{2^n}
\]
forces $|T|\ge2^n$. Hence $|T|=2^n$ and the above estimate is an equality, which means that
\begin{equation}\label{eq:tiling}
	G=\bigcup_{t\in T}tA\quad\text{is a disjoint union.}
\end{equation}

For every $t\in T$ pick $y_t\in tA$. By~\eqref{eq:tiling} we have $y_t\notin sA$ for $s\in T\setminus\{t\}$, so~\eqref{eq:vanish} yields
\[
	p_{y_t}(\alpha_t)\ne0,
	\qquad
	p_{y_t}(\alpha_s)=0\quad(s\in T\setminus\{t\}).
\]
As in the proof of \autoref{lem:helly} the polynomials $p_{y_t}$ with $t\in T$ are linearly independent, and since there are $2^n=\dim\mathcal M_n$ of them, they form a basis of $\mathcal M_n$.

Let $t_0\in T$ be fixed, and let $y\in t_0A$ be arbitrary. Again by~\eqref{eq:tiling} and~\eqref{eq:vanish} we have $p_y(\alpha_s)=0$ for $s\in T\setminus\{t_0\}$. Writing $p_y=\sum_{t\in T}c_tp_{y_t}$ and evaluating at $\alpha_s$ gives $c_sp_{y_s}(\alpha_s)=p_y(\alpha_s)=0$, hence $c_s=0$ for every $s\ne t_0$. Thus $p_y=c\,p_{y_{t_0}}$ for some $c\in\QQ$. Comparing in~\eqref{eq:pyalpha} the coefficients of $X_1\cdots X_n$ gives $c=1$, and comparing the coefficients of $\prod_{i\ne j}X_i$ gives
\[
	\lambda_j(yH_j)=\lambda_j(y_{t_0}H_j)\qquad(1\le j\le n).
\]
Since $\lambda_j$ is injective, this means $yH_j=y_{t_0}H_j$ for all $j$, that is, $y\in y_{t_0}K$. As $y\in t_0A$ was arbitrary, $t_0A\subseteq y_{t_0}K$ and therefore
\[
	|A|=|t_0A|\le|K|.
\]
Combined with $K\subseteq A$ this gives $A=K$, and $|G:K|=|G|/|A|=2^n$. 

If $H_i=H_j$ for some $i\ne j$, then $\sim_i=\sim_j$ and \autoref{lem:helly} gives $|T|<2^n$, contradicting $|T|=2^n$. Hence the $H_i$ are pairwise distinct.
\end{proof}

\section{Irredundant covers}\label{sec:cons}

\begin{proof}[Proof of \autoref{thm:B}]
Fix $i\in[n]$ and put $A_i:=G\setminus\bigcup_{j\ne i}C_j$. Irredundancy gives $A_i\ne\varnothing$ and the covering property gives $A_i\subseteq C_i$. Applying \autoref{thm:A} to the $n-1$ cosets $C_j$ with $j\ne i$ we find $t_1,\ldots,t_m\in G$ with $m\le2^{n-1}$ and
\[
	G=\bigcup_{k=1}^mt_kA_i\subseteq\bigcup_{k=1}^mt_kg_iH_i.
\]
Thus $G$ is a union of at most $2^{n-1}$ left cosets of $H_i$ and $|G:H_i|\le2^{n-1}$.

For the sharpness let $G:=\FF_2^{n-1}$, let $C_i:=\{x\in G:x_i=1\}$ for $i<n$ and let $C_n:=\{0\}$. Every nonzero $x\in G$ has a coordinate equal to $1$, so these $n$ cosets cover $G$. The $i$-th standard basis vector lies in $C_i$ only, and $0$ lies in $C_n$ only, so the cover is irredundant. Here $H_n=1$ and $|G:H_n|=2^{n-1}$.
\end{proof}

\section*{Acknowledgment}
The large language models GPT-5.6-sol and Claude Opus~5 assisted in the preparation of this paper. I have checked all details and take full responsibility for the contents of this paper.

\end{document}